\documentclass[11pt,reqno]{amsart}

\usepackage{amsmath,amsfonts,amstext,amsxtra,amssymb,amsthm,mathtools}
\usepackage{array}
\usepackage[utf8]{inputenc}
\usepackage[english]{babel}

\usepackage[backend=biber,style=alphabetic,sorting=anyt,maxnames=99]{biblatex}
\usepackage{csquotes}
\DeclareFieldFormat*{title}{\mkbibemph{#1}}

\DeclareFieldFormat*{volume}{\mkbibbold{#1}}
\DeclareFieldFormat*{number}{\mkbibbold{#1}}

\usepackage[left=2.5cm,paperwidth=21cm,right=2.5cm,paperheight=29.7cm,textheight=23.5cm,top=2.5cm]{geometry}

\usepackage{tikz}
\usetikzlibrary{matrix,quotes,cd,decorations.pathmorphing,shapes}

\usepackage{enumitem}
\usepackage[hidelinks,colorlinks,citecolor={blue!70!black},urlcolor=black,linkcolor={blue!90!black}]{hyperref}
\usepackage[noabbrev]{cleveref}
\usepackage{xfrac}
\usepackage{xcolor}

\newtheorem{definition}{Definition}[section]

\newtheorem{proposition}[definition]{Proposition}
\newtheorem{cor}[definition]{Corollary}
\newtheorem{lemma}[definition]{Lemma}
\newtheorem{theorem}[definition]{Theorem}
\newtheorem*{theorem*}{Theorem}
\newtheorem*{definition*}{Definition}

\newtheorem{remark}[definition]{Remark}

\theoremstyle{remark}

\allowdisplaybreaks

\numberwithin{equation}{section}

\title{The Diagonal of (3,3) fivefolds}
\author[Jan Lange]{Jan Lange}
\address{Institute of Algebraic Geometry \\ Leibniz University Hannover \\ Welfengarten 1 \\ 30167 Hannover, Germany}
\email{\href{mailto:lange@math.uni-hannover.de}{lange@math.uni-hannover.de}}

\author[Bj{\o}rn Skauli]{Bj{\o}rn Skauli}
\address{Department of Mathematics \\ University of Oslo \\ Moltke Moes vei 35 0851 Oslo, Norway}
\email{\href{mailto:bjorska@math.uio.no}{bjorska@math.uio.no}}

\DeclareMathOperator{\blowup}{Bl}
\DeclareMathOperator{\CH}{CH}
\DeclareMathOperator{\charac}{char}
\DeclareMathOperator{\divisor}{div}

\DeclareMathOperator{\Ima}{Im}
\DeclareMathOperator{\Ker}{Ker}

\DeclareMathOperator{\Sing}{Sing}
\DeclareMathOperator{\specialization}{sp}
\DeclareMathOperator{\Spec}{Spec}
\DeclareMathOperator{\trdeg}{tr.deg}

\newcommand{\Z}{\mathbb{Z}}
\newcommand{\C}{\mathbb{C}}
\newcommand{\projspace}{\mathbb{P}}

\newcommand{\restr}[1]{\left. #1 \right|} 
\newcommand{\vof}[1]{\left| #1 \right|}
\newcommand{\innerpro}[2]{\left\langle #1 , #2 \right\rangle}

\begin{document}

\begin{abstract}
        We show that a very general (3,3) complete intersection in $\projspace^7$ over an uncountable algebraically closed field of characteristic different from $2$ admits no decomposition of the diagonal, in particular it is not retract rational. This strengthens Nicaise and Ottem's result in \cite{NO22} where stable irrationality in characteristic 0 was shown. The main tool is a Chow-theoretic obstruction which was found by Pavic and Schreieder in \cite{PS23}, where quartic fivefolds are studied.
\end{abstract}

\date{February 28, 2024}

\maketitle

\section{Introduction}
The Lüroth problem asks whether rationality and unirationality are equivalent. This holds for curves and complex surfaces, but in higher dimensions the two notions start to differ. Some interesting intermediate properties between rational and unirational have therefore been introduced, but the relations between these are not fully understood yet.

Recall that a variety $X$ is \emph{rational} if it is birational to projective space, and \emph{stably rational} if $X \times \mathbb{P}^m$ is rational for some $m$. We say that $X$ is \emph{retract rational} if the identity map of $X$ factors through some projective space as a rational map, i.e. there exists rational maps $f$ and $g$ such that the composition
$$
    X \overset{f}{\dashrightarrow} \projspace^n \overset{g}{\dashrightarrow} X
$$
is defined and equal to the identity on a nonempty open set $U \subset X$. Finally, a variety $X$ is called \emph{unirational} if there is a dominant map $\projspace^n \dashrightarrow X$. There are straightforward implications
$$
    \text{rational} \Longrightarrow \text{stably rational} \Longrightarrow \text{retract rational} \Longrightarrow \text{unirational}.
$$
Over algebraically closed fields only the first and third implications are known to be strict. Beauville, Colliot-Thélène, Sansuc, and Swinnerton-Dyer \cite{BCTSSD85} showed the strictness of the first implication over $\C$. The first counterexample over $\C$ to the third implication was constructed by Artin and Mumford in \cite{AM72}. Over non-closed fields, there are algebraic tori which are retract rational, but not stably rational. An overview of these can be found in \cite{HY17}. But the problem remains open for algebraically closed fields.

Voisin \cite{Voi15} introduced a cycle-theoretic specialization technique to prove retract irrationality based on the decomposition of the diagonal, (see \Cref{sec:Decomposition}) and applied it to the very general quartic double solid. This specialization technique was later generalized and refined by Colliot-Thélène and Pirutka in \cite{CTP16} and by Schreieder in \cite{Sch19b}. Totaro (\cite{Tot16}) used this technique to improve Kollár's result (\cite{Kol95}) on the irrationality of very general hypersurfaces of degree greater than roughly $\frac{2}{3}$ of the dimension to stable irrationality. Later on, Schreieder achieved a logarithmic bound for retract irrationality in \cite{Sch19b}. Beyond hypersurfaces, rationality of complete intersections has also been studied using the same technique by e.g. Chatzistamatiou and Levine in \cite{ChatzistamatiouLevine} and Hassett, Pirutka and Tschinkel in \cite{HPT18b}.
The latter three authors also used the technique to study how rationality can vary in families \cite{HPT18}.

In characteristic 0, a different approach was introduced by Nicaise and Shinder \cite{NS19}, as well as Kontsevich and Tschinkel \cite{KT19}. This method is based on motivic integration and the weak factorization theorem and provides conditions under which stable rationality is preserved under specialization. Using this, Nicaise and Ottem \cite{NO22} proved stable irrationality of quartic fivefolds and complete intersections of two cubics in $\projspace^7$ and also used the results from \cite{Sch19b} to prove stable irrationality of many other complete intersections.

With this technique, the strongest results typically arise by specializing into a union of several components such that some of the components intersect in a lower dimensional variety, and this intersection is known to be stable irrational by some other method. Since this method does not a priori obstruct retract rationality, it is an interesting question if it can be used to find retract rational but stably irrational varieties over algebraically closed fields.

In \cite{PS23}, Pavic and Schreieder introduce a Chow-theoretic analogue of the motivic method, where retract rationality can be obstructed by degenerating to a union where the obstruction to a decomposition of the diagonal, and hence to rationality, lies in the intersection of the components. Using this technique, they study the very general quartic fivefold and show that it is also retract irrational. Despite the analogy, the relation between the method in \cite{PS23} and the one used in \cite{NO22} is unclear, and results obtained by one might not necessarily translate to the other. Hence it is worthwhile to study to what extent Pavic and Schreieder's method applies to the new examples of stably irrational varieties found in \cite{NO22}.

In this paper, we study the (3,3) fivefold example from \cite{NO22}, i.e. a complete intersection of two cubics in $\projspace^7$, and apply the method of \cite{PS23} to show that it is also retract irrational. More precisely, we show the following result.

\begin{theorem}\label{thm:MainResultIntro}
    Let $k$ be an uncountable algebraically closed field of characteristic different from $2$. Then the very general (3,3) fivefold over $k$ admits no decomposition of the diagonal, in particular it is not retract rational.
\end{theorem}
Over fields of positive characteristic, the rationality of the very general (3,3) fivefold was previously open. Additionally, we present an explicit example of a retract irrational (3,3) fivefold.

The paper is organized as follows: In \Cref{sec:Preliminaries} we recall the Chow-theoretic technique of Pavic and Schreieder as well as some basic definitions of e.g. decomposition of the diagonal and specializations. The main result is then proven in \Cref{sec:Main}, which is split up in four parts: First we follow Nicaise and Ottem \cite[Theorem 7.2]{NO22} and degenerate the complete intersection to a union of two components with a carefully chosen intersection. We then follow the methods in \cite{PS23} to obtain a specialization of our complete intersection to a union of three components. The components are then simplified using further specializations, letting us apply the obstruction found in \cite{PS23}. In the third part we combine this with the fact that one of the three components is chosen to be stably birational to the quadric surface bundle described by Hassett, Pirutka and Tschinkel in \cite{HPT18} to obstruct the existence of a decomposition of the diagonal. Finally, we show that \Cref{thm:MainResultIntro} follows from this.

\subsection*{Acknowledgement}
    The authors are grateful to John Ottem, Nebojsa Pavic and Stefan Schreieder for suggesting this collaboration and also for their patience in answering questions. The authors would also like to thank the referee for their careful reading of the paper, and for their helpful comments. The first named author received funding from the European Research Council (ERC) under the European Union's Horizon 2020 research and innovations programe under grant agreement No. 948066 (ERC - StG RationAlgic).

\section{Preliminaries}
\label{sec:Preliminaries}

\subsection{Conventions and Notations}
    An \emph{algebraic $k$-scheme} is a separated scheme of finite type over a field $k$. A \emph{$k$-variety} (or \emph{variety}) is an integral, algebraic $k$-scheme. Let $X$ be a $k$-variety. We denote the \emph{function field} of $X$ by $k(X)$ and the \emph{residue field} of a closed point $x \in X$ by $\kappa(x)$. For a separated scheme $X$ over a ring $R$ and some ring extension $A/R$ we write \emph{$X_A := X \times A := X \times_R A := X \times_{\Spec R} \Spec A$} for the base change.
    
    We denote the \emph{Chow group of $l$-cycles} of a $k$-variety $X$ by $\CH_l(X)$, which is the quotient of the free abelian group generated by $l$-dimensional subvarieties modulo rational equivalence.
    
    A \emph{very general point} of an irreducible separated scheme is a closed point in the complement of a countable union of proper closed subsets.
    
    We will write \emph{$(d_1,\dots,d_k)$ $n$-fold} or \emph{$(d_1,\dots,d_k)$ complete intersection} for the intersection of $k$ hypersurfaces in $\projspace^{n+k}$ of degree $d_1,\dots,d_k$.

\subsection{Decomposition of the diagonal}
    \label{sec:Decomposition}
    We briefly introduce the notion of (Chow-theoretic) decomposition of the diagonal and its relation to rationality questions.
    
    Let $X$ be a $k$-variety of dimension $n$ and let $\Delta_X \subset X \times_k X$ be the diagonal. We say that $X$ admits a  (Chow-theoretic) decomposition of the diagonal if there exists a zero-cycle $z$ on $X$ and an $n$-cycle $Z_X \subset X \times_k X$, which does not dominate the first factor, such that
    $$
        [\Delta_X] = [X \times_k z] + [Z_X] \in \CH_n(X \times_k X).
    $$
    Here $[\cdot]$ denotes the class of the cycle in the Chow group. Pulling back the diagonal $\Delta_X$ via the natural morphism $X_{k(X)} \to X \times_k X$ yields a zero-cycle $\delta_X \in \CH_0(X_{k(X)})$. Then $X$ admits a decomposition of the diagonal if and only if there is an equality
    $$
        \delta_X = [z_{k(X)}] \in \CH_0(X_{k(X)}),
    $$
    for some zero-cycle $z$ on $X$, see e.g. \cite[Lemma 7.3]{Sch21}.
    
    Our interest in a decomposition of the diagonal comes from the following lemma, which is not hard to prove, see e.g. \cite[Lemma 7.5]{Sch21}.
    \begin{lemma}
        A retract rational $k$-variety admits a decomposition of the diagonal.
    \end{lemma}
    
    We say that a proper $k$-variety $X$ has \emph{universally trivial Chow group of zero-cycles} (short: universally trivial $\CH_0$) or \emph{$X$ is universally $\CH_0$-trivial} if for any field extension $F/k$, the degree map
    $$
        \deg \colon \CH_0(X_F) \longrightarrow \Z
    $$
    is an isomorphism. From this definition it is obvious that varieties with universally trivial $\CH_0$ admit a decomposition of the diagonal. The converse holds for geometrically integral and smooth $k$-varieties, see \cite[Proposition 1.4]{CTP16}.

\subsection{Chow-theoretic obstruction to retract rationality}

    In this section we recall the constructions of the obstruction map from \cite[Section 3]{PS23}. Throughout this section $R$ denotes a discrete valuation ring with residue field $k$ and fraction field $K$. A proper flat $R$-scheme $\mathcal{X}$ is called \emph{strictly semi-stable} if the special fibre $X_k = \mathcal{X} \times_R k$ is a geometrically reduced simple normal crossing divisor on $\mathcal{X}$. In other words, the components $Y_i$ ($i \in \{1,\dots,m\}$) of $X_k$ are smooth Cartier divisors in $\mathcal{X}$ and the scheme-theoretic intersection $\bigcap\limits_{j \in J} Y_j$ is smooth of codimension $\vof{J}$ or empty for every $J \subset \{1,\dots,m\}$.
    
    \begin{definition}[{\cite[Definition 3.1]{PS23}}]\label{def:ObstructionMap}
        Let $\iota \colon X_k \hookrightarrow \mathcal{X}$ and $\iota_i \colon Y_i \hookrightarrow \mathcal{X}$ denote the natural embeddings. For every $i \in \{1,\dots,m\}$ define
        $$
            \Phi_{\mathcal{X},Y_i} \colon \CH_1(X_k) \overset{\iota_\ast}{\longrightarrow} \CH_1(\mathcal{X}) \overset{\iota_i^\ast}{\longrightarrow} \CH_0(Y_i)
        $$
        to be the composition of the push-forward along the embedding $\iota$ and the intersection with the Cartier divisor $Y_i \subset \mathcal{X}$. We denote the direct sum by
        \begin{equation}\label{eq:ObstructionMap}
            \Phi_{\mathcal{X}} := \sum\limits_{i = 1}^m \Phi_{\mathcal{X},Y_i} \colon \CH_1(X_k) \longrightarrow \bigoplus\limits_{i=1}^m \CH_0(Y_i).
        \end{equation}
    \end{definition}
    
    Although the involved Chow groups depend only on the special fibre $X_k$, the obstruction map might a priori depend on the choice of the strictly semi-stable family. We recall the explicit description of $\Phi$ in \cite[Lemma 3.2]{PS23}, which shows that $\Phi_\mathcal{X}$ in fact depends only on the special fibre $X_k$, and not on the total space $\mathcal{X}$.
    
    \begin{lemma}
        With the same notation as in \Cref{def:ObstructionMap}, let furthermore $Y_{i,j}:= Y_i \cap Y_j$ be the scheme-theoretic intersection of two components $Y_i$ and $Y_j$ of the special fibre $X_k$ and denote by $\iota_{i,j} \colon Y_{i,j} \to Y_j$ and $\iota_i \colon Y_i \to X_k$ the natural inclusions. Moreover, we write $\restr{\gamma_i}_{Y_{j,i}} := \iota_{j,i}^\ast \gamma_i$ for the restriction of $\gamma_i \in \CH_1(Y_i)$ to the intersection $Y_{i,j}$. Then for any $\gamma_j \in \CH_1(Y_j)$:
        \begin{equation*}
            \Phi_{\mathcal{X},Y_i}\left(\left(\iota_j \right)_\ast \gamma_j\right) = \begin{cases}
                \left(\iota_{j,i}\right)_\ast \left(\restr{\gamma_j}_{Y_{i,j}}\right) & \text{for } j \neq i, \\
                - \sum\limits_{k \neq j} \left(\iota_{k,i}\right)_\ast \left(\restr{\gamma_j}_{Y_{k,j}}\right) & \text{for } j = i.
            \end{cases}
        \end{equation*}
        In particular, for $\gamma = \sum\limits_{k=1}^m \left(\iota_k\right)_\ast \gamma_k \in \CH_1(X_k)$ and $i \in \{1,\dots,m\}$:
        \begin{equation}\label{eq:ExplicitDescriptionOfObstructionMap}
            \Phi_{\mathcal{X},Y_i}(\gamma) = \sum\limits_{j \neq i} (\iota_{j,i})_\ast \restr{\gamma_j}_{Y_{i,j}} - \sum\limits_{j \neq i} (\iota_{j,i})_\ast \restr{\gamma_i}_{Y_{j,i}}.
        \end{equation}
    \end{lemma}
    
     In \cite{PS23}, Pavic and Schreieder make two additional observations about this obstruction map $\Phi_{\mathcal{X}}$, which we also recall here: Firstly, the concrete description of $\Phi_{\mathcal{X}}$ in \eqref{eq:ExplicitDescriptionOfObstructionMap} allows us to conclude that $\deg \Phi_\mathcal{X}(\gamma) = 0$ for all $\gamma \in \CH_1(X_k)$. Secondly, for any unramified extension of DVR's $A/R$, i.e. $R \to A$ injective and local morphism of DVR's with $m_R \cdot A = m_A$, the base-change $\mathcal{X}_A = \mathcal{X} \times_R A$ is a strictly semi-stable $A$-scheme. Indeed, since $A/R$ is unramified, the residue field $L$ of $A$ is isomorphic to $k \otimes_ R A$, i.e. the special fibre of $\mathcal{X}_A \to \Spec R$ is the base-extension with $L$ of the special fibre $X_k$ of $\mathcal{X} \to \Spec R$. Thus we get for any unramified extension $A/R$ with residue field $L$ an obstruction map
    $$
        \Phi_{\mathcal{X}_A} \colon \CH_1(X_L) \longrightarrow \Ker \left(\bigoplus\limits_{i=1}^m \CH_0(Y_{i,L}) \overset{\deg}{\longrightarrow} \Z\right).
    $$
    Studying these maps can give an obstruction to the decomposition of the diagonal of the geometric generic fibre.
    
    \begin{theorem}[{\cite[Theorem 4.1]{PS23}}]\label{thm:PavicSchreieder}
        Let $R$ be a discrete valuation ring with algebraically closed residue field $k$ and let $\mathcal{X} \to \Spec R$ be a strictly semi-stable projective $R$-scheme whose special fibre $X_k = \bigcup\limits_{i = 1}^m Y_i$ is a chain of Cartier divisors, i.e. only neighboring components intersect non-trivially. Assume that the geometric generic fibre of $\mathcal{X} \to \Spec R$ has a decomposition of the diagonal. Then for any unramified extension $A/R$ of DVR, with induced extension $L/k$ of residue fields, the natural map
        $$
            \Phi_{\mathcal{X}_A} \colon \CH_1(X_L)/2 \longrightarrow \Ker \left(\bigoplus\limits_{i = 1}^m \CH_0(Y_{i,L})/2 \overset{\deg}{\longrightarrow} \Z/2\right)
        $$
        is surjective.
    \end{theorem}    
    \begin{remark}\label{rem:SlightSimplificationOfPavicSchreieder}
        If $R = k[[t]]$ is the formal power series in one variable, then \Cref{thm:PavicSchreieder} can be reformulated to the following because for every field extension $L/k$ the DVR $A = L[[t]]$ is an unramified extension of $R$: If the geometric generic fibre of $\mathcal{X} \to \Spec R$ admits a decomposition of the diagonal, then the morphism $$
            \Phi_{\mathcal{X}} \colon \CH_1(X_k) \longrightarrow \Ker \left(\bigoplus\limits_{i=1}^m \CH_0(Y_{i}) \overset{\deg}{\longrightarrow} \Z\right)
        $$
        is universally surjective modulo $2$.
    \end{remark}

\subsection{Specializations}

We will often use specializations of varieties, or more generally reduced algebraic schemes, not only to obstruct rationality but also to show that certain varieties are smooth. We therefore introduce this notion here, following \cite[Section 2.2]{Sch19a}. Let $Y$ and $Z$ be reduced algebraic schemes over a field $F$ and an algebraically closed field $k$, respectively. We say that \emph{$Y$ specializes (or degenerates) to $Z$} if there exists a DVR $R$ with residue field $k$ and fraction field $K$, together with an injection of fields $K \to F$ such that the following holds: There exists a proper, flat morphism $\mathcal{X} \to \Spec R$ of finite type such that $Z$ is isomorphic to its special fibre $X_k = \mathcal{X} \times_{R} k$ and $Y$ is isomorphic to the base change $X_F = (\mathcal{X} \times_R K) \times_K F$ of the generic fibre.

For the convenience of the reader and to refer to later, we sketch two well-known arguments. First we recall that it suffices to check smoothness after some proper specialization.

\begin{remark}\label{rem:CheckSmoothnessAfterSpecialization}
    With the above notation, we claim that $Y$ is smooth if $Z$ is smooth. Indeed, since smoothness is stable under extension of the base field, it suffices to check that the generic fibre $X_K$ is smooth if the special fibre $X_k$ is smooth. But this follows directly from the facts that the morphism $\mathcal{X} \to \Spec R$ is proper and that being singular is a closed condition. So we find that it suffices to check smoothness after some proper specialization.
\end{remark}

Secondly, we recall the following, see e.g. \cite[Lemma 8]{Sch19a}.

\begin{lemma}\label{lem:VeryGeneralSpecializesToOne}
    Let $f \colon \mathcal{X} \to B$ be a surjective, proper, and flat morphism of reduced, quasi-projective algebraic schemes over an algebraically closed field $k$ and assume further that $B$ is integral. Let $0 \in B$ be a closed point. Then a very general fibre specializes to the fibre $X_0$ over the point $0$ in the above sense.
\end{lemma}

\begin{proof}
    A very general fibre of $f$ is abstractly isomorphic to the geometric generic fibre of $f$, see e.g. \cite[Lemma 2.1]{Vi13}. Hence it suffices to show that one very general fibre specializes to $X_0$. Fix one very general fibre $X_b$, then a very general point on a curve in $B$ through $0$ and $b$ is also a very general point of $B$, i.e. we can reduce to the case where $B$ is an (integral) curve. After passing to the normalization, we can assume further that $B$ is smooth. Thus the local ring $\mathcal{O}_{B,0}$ at $0 \in B$ is an integrally closed Noetherian local ring, i.e. a DVR, proving the lemma.
\end{proof}

By Fulton's specialization map (\cite[§20.3]{Ful98}), having a decomposition of the diagonal behaves well with respect to specializations, cf. \cite[Theorem 1.14]{CTP16}. Thus we can prove \Cref{thm:MainResultIntro} by constructing one example of a (3,3) fivefold with no decomposition of the diagonal. This is the main point in \Cref{sec:Main}. The details on why this suffices to prove \Cref{thm:MainResultIntro} can be found at the end of this paper (\Cref{cor:MainResult}).

\subsection{Alterations}
For a variety $W$, an alteration is a smooth variety $\tilde{W}$ with a surjective, proper and generically finite map to $W$. By work of de Jong \cite{deJ96}, alterations always exist, also in positive characteristic. Moreover, by work of Gabber (see \cite{IT14}), we can choose the alteration to have odd degree if the characteristic of the ground field is different from 2.

\section{Very general (3,3) fivefolds are irrational}
\label{sec:Main}
    We aim to construct an explicit example of a (3,3) complete intersection in $\projspace^7$, which does not admit a decomposition of the diagonal. Before going into the details we sketch the construction. We start with a complete intersection
    $$
        X = X_1 \cap X_2 \subset \projspace^7
    $$
    of two cubic hypersurfaces $X_1$ and $X_2$. Following \cite[Theorem 7.2]{NO22}, we specialize $X_2$ into a union of a hyperplane and a quadric hypersurface in $\projspace^7$. This specialization yields a family $\mathcal{X} \to \Spec k[[t]]$ with $X$ as the geometric generic fibre. The special fibre of that family is the union of a cubic hypersurface $Y \subset \projspace^6$ and a (2,3) complete intersection $Z \subset \projspace^7$ meeting at a (2,3) fourfold $W$. This family is not strictly semi-stable because the total space $\mathcal{X}$ is singular. Moreover, the obstruction to rationality used in \cite[Theorem 7.2]{NO22} lies in $W$ and is thus not seen by the obstruction morphism \eqref{eq:ObstructionMap}. By blowing up one component of the special fibre, we ensure that the family is semi-stable. We can then blow up $W$ to introduce a top-dimensional component stably birational to $W$, which is seen by the obstruction morphism.
    
    We end up with a strictly semi-stable family $\Tilde{\mathcal{X}} \to \Spec k[[t]]$. Its special fibre has three irreducible components $\Tilde{Y}$, $P_W$ and $Z$. Here $\Tilde{Y}$ is the blow-up of $Y$ in the singular locus $S$ of the total space $\mathcal{X}$ and $P_W$ is a $\projspace^1$-bundle over $W$. The details of this construction are presented in the next section.
    
    We aim to show that the homomorphism
    $$
        \Phi_{\Tilde{\mathcal{X}}} \colon \CH_1(\Tilde{X}_k) \longrightarrow \Ker \left(\CH_0(\Tilde{Y}) \oplus \CH_0(P_W) \oplus \CH_0(Z) \overset{\deg}{\longrightarrow} \Z \right)
    $$
    is not universally surjective modulo $2$. It then follows from \Cref{thm:PavicSchreieder} and \Cref{rem:SlightSimplificationOfPavicSchreieder} that the geometric generic fibre of $\Tilde{\mathcal{X}} \to \Spec R$ admits no decomposition of the diagonal. Aiming for a contradiction we assume that the homomorphism is universally surjective modulo $2$, in particular that the zero-cycle \begin{equation}\label{eq:ZeroCycleDiagonal-Point}
        \delta_{P_W} - z_{k(P_W)} \in \CH_0(P_{W,k(P_W)})
    \end{equation}
    is contained in the image of $\Phi_{\Tilde{\mathcal{X}}}$, basechanged to $k(P_W)$, modulo $2$. To find a contradiction we need to control the image of $\Phi_{\Tilde{\mathcal{X}}}$, in particular the involved Chow groups. To do this we follow the approach in \cite{PS23}. The  Chow groups are hard to compute, but by using Fulton's specialization map (\Cref{lem:FultonSpecialization}) we can specialize further to control the Chow groups. We conclude that if the zero cycle \eqref{eq:ZeroCycleDiagonal-Point} is contained in the image of $\Phi_{\Tilde{\mathcal{X}}}$ modulo $2$, then it must also be contained in the image of the natural homomorphism
    \begin{equation}\label{eq:NaturalExtensionOfCH0OfPW}
        \CH_0(P_W)/2 \longrightarrow \CH_0(P_{W,k(P_W)})/2.
    \end{equation}
    To obtain the final contradiction, we note that $W$ is chosen to be birational to the quadric surface bundle studied by Hassett, Pirutka and Tschinkel in \cite{HPT18} and therefore has a nonzero unramified cohomology class. By computing the Merkurjev pairing of this class and the diagonal class $\delta_{P_W}$, and comparing it with the pairing with classes in the image of \eqref{eq:NaturalExtensionOfCH0OfPW}, we obtain a contradiction, proving that $\Phi$ is not universally surjective.

    \subsection{A strictly semi-stable family}\label{sec:sssfamily}
    We construct a strictly semi-stable family as outlined above. Let $k_0$ be an algebraically closed field of characteristic different from $2$, and let \begin{equation}\label{eq:k}
        k = \overline{k_0(\alpha,\beta,\gamma)}
    \end{equation}
    be the algebraic closure of a purely transcendental field extension of transcendence degree $3$ over $k_0$. The parameters $\alpha,\beta$ and $\gamma$ allow us to degenerate the involved varieties in order to make the Chow groups more accessible. To obtain the existence of a non-trivial unramified cohomology class we need to consider the following two polynomials.

    \begin{definition}
        We define the following polynomials in $k_0[x_0,\dots,x_6]$:
        \begin{align*}
            c_0 &= x_0^2 x_5 + x_1^2 x_4 + x_2^2 x_6 + x_3 (x_3^2 + x_4^2 + x_5^2 - 2x_3(x_4 + x_5 + x_6)), \\
            q_0 &= x_3 x_6 - x_4 x_5.
        \end{align*}
    \end{definition}

    The (2,3)-fourfold $W_0$ given by the vanishing of these two polynomials is precisely the one studied in \cite{Ska23}, and is birational to the quadric surface bundle studied in \cite{HPT18}, see \cite[Corollary 3.8]{Ska23}. Moreover, the second named author showed in \cite{Ska23} that $W_0$ does not admit a decomposition of the diagonal.
    
    We need to pick the exact equations defining the specialization we use with some care. To use \Cref{thm:PavicSchreieder} we need the specialization to be strictly semi-stable, and in particular the components of the special fibre must be smooth. However, we must also ensure that we can specialize further afterwards to simplify the Chow groups. Finally, we must take care that even after this specialization, we can use what we know about $W_0$ to obtain a contradiction. We therefore choose polynomials in the following way:
    
    \begin{definition}\label{def:polynomials}
        Let $k$ be defined as in \eqref{eq:k}. Consider the following polynomials in $k[x_0,\dots,x_7]$ with $\charac k \neq 2$:
        \begin{align*}
            c_{\alpha,\beta,\gamma} &:= c_0 + \gamma \left(x_6 p_3 + c_3\right) + \beta c_2 + \alpha c_{1}, \\
            q_{\alpha,\beta} &:= q_0 + \beta (x_3 x_7 + q_2) + \alpha q_1, \\
            f_{\alpha,\beta} &:= x_6^3 + \beta f_2 + \alpha f_1,
        \end{align*}
        where $p_3, c_3 \in k[x_0,\dots,x_5]$, $c_2,f_2,q_2 \in k[x_0,\dots,x_6]$, and $c_1,f_1,q_1 \in k[x_0,\dots,x_7]$ are general polynomials of degree
        $$
            \deg c_1 = \deg c_2 = \deg c_3 = \deg f_1 = \deg f_2  = 3, \ \deg q_1 = \deg q_2 = \deg p_3 = 2,
        $$
        i.e. $c_{\alpha,\beta,\gamma}$ and $f_{\alpha,\beta}$ are homogenous polynomials of degree 3 and $q_{\alpha,\beta}$ is a homogenous polynomial of degree 2. Addionally we choose them such that the following varieties are smooth:
        \begin{align*}
            &\left\{c_3 = 0\right\}, \ \left\{c_3 = p_3 = 0\right\} \subset \projspace^5, \\
            &\left\{c_2 = 0\right\}, \ \left\{c_2 = f_2 = 0\right\}, \ \left\{c_2 = q_2 = 0\right\}, \ \left\{c_2 = q_2 = x_3 = 0\right\} \subset \projspace^6, \\
            &\left\{c_1 = 0\right\}, \ \left\{c_1 = q_1 = 0\right\}, \ \left\{c_1 = f_1 = 0\right\}, \ \left\{c_1 = q_1 = f_1 = x_7 = 0\right\} \subset \projspace^7.
        \end{align*}
        Moreover, the generality assumption includes that the varieties \begin{equation}\label{eq:weakerassumption}
            \{q_1 = 0\}, \ \{f_1 = 0\} \subset \projspace^7_k \text{ are smooth along } \{c_1 = f_1 = q_1 = x_7 = 0\} \subset \projspace^7.
        \end{equation}
    \end{definition}

    \begin{remark}
        \label{rem:ExplicitEquations}
        The existence of such polynomials follows from Bertini's theorem. Note that the generality condition \eqref{eq:weakerassumption} can be replaced by the stronger condition that $\{q_1 = 0\}, \ \{f_1 = 0\} \subset \projspace^7$ are smooth to make the Bertini-type argument for the existence more immediate.
        
        For completeness, we include a possible specific choice of polyonomials. In $\charac k \neq 3$, we can pick the following:
        \begin{align*}
            p_3 &:= \sqrt[3]{4} (x_1x_2 + x_4x_5) + x_3^2 \in k_0[x_0,\dots,x_5], \\
            c_3 &:= x_0^3 + x_1^3 + x_2^3 + x_3^3 + x_4^3 + x_5^3 \in k_0[x_0,\dots,x_5], \\
            c_2 &:= c_3 + x_6^3 \in k_0[x_0,\dots,x_6], \\
            c_1 &:= p_1 + x_7^3 \in k_0[x_0,\dots,x_7], \\
            q_2 &:= p_3 + x_6^2 \in k_0[x_0,\dots,x_6], \\
            f_2 &:= x_0^3 - x_1^3 + \rho x_2^3 - \rho x_3^3 + \rho^2 x_4^2 - \rho^2 x_5^2 \in k_0[x_0,\dots,x_5],
        \end{align*}
        where $\rho \in k_0$ is a primitive third root of unity and $\sqrt[3]{4} \in k_0$ is a cube root of $4$ and $p_1, \ q_1,$ and $f_1$ are general polynomials in $k[x_0,\dots,x_6]$ such that the hypersurfaces cut out by these polynomials as well as all their intersections are smooth;
        
        In $\charac k = 3$ we can pick,
        \begin{align*}
            p_3 &:= x_1^2 + x_2^2 + x_3^2 + x_4^2 + x_5^2 \in k_0[x_0,\dots,x_5], \\
            c_3 &:= x_0^3 + x_0 x_1^2 + x_1 x_2^2 + x_2 x_4^2 + x_4 x_5^2 + x_5 x_3^2 \in k_0[x_0,\dots,x_5], \\
            c_2 &:= c_3 + x_3 x_6^2 \in k_0[x_0,\dots,x_6], \\
            c_1 &:= c_2 + x_6  x_7^2 \in k_0[x_0,\dots,x_7], \\
            q_2 &:= p_3 - x_6^2 \in k_0[x_0,\dots,x_6], \\
            q_1 &:= q_2 + x_7^2 \in k_0[x_0,\dots,x_7], \\
            f_2 &:= x_1^2 x_2 + x_2^2 x_4 + x_4^2 x_5 + x_5^2 x_3 + x_3^2 x_6 + x_6^3 \in k_0[x_0,\dots,x_6], \\
            f_1 &:= f_2 + x_6^2 x_7 + x_7^3 \in k_0[x_0,\dots,x_7].
        \end{align*}
        Note that with these choices the varieties $\{q_1 = 0\}$ and $\{f_1 = 0\} \subset \projspace^7$ are not smooth, but they satisfy the weaker assumption \eqref{eq:weakerassumption}.
    \end{remark}

    We construct the model $\mathcal{X} \to \Spec k[[t]]$, inspired by \cite[Theorem 7.2]{NO22}. Let $R := k[[t]]$ and consider the $R$-scheme
    $$
        \mathcal{X} := \left\{c_{\alpha,\beta,\gamma} = t f_{\alpha,\beta} + x_7 q_{\alpha,\beta} = 0\right\} \subset \projspace^7_R.
    $$
    The special fibre $X_k$ of $\mathcal{X} \to \Spec R$ has two irreducible components, namely a cubic fivefold $Y := \{c_{\alpha,\beta,\gamma} = x_7 = 0\} \subset \projspace^7_k$ and a (2,3) complete intersection $Z := \{c_{\alpha,\beta,\gamma} = q_{\alpha,\beta} = 0\} \subset \projspace^7_k$. We denote their scheme-theoretic intersection by $W := Y \cap Z$.

    \begin{lemma}\label{lem:ComponentsOfSpecialFibreAreSmooth}
        The varieties $Y,\ Z,$ and $W$ in $\projspace^7$ are smooth. The geometric generic fibre $$
            X_{\overline{K}} = \left\{c_{\alpha,\beta,\gamma} = f_{\alpha,\beta} + t^{-1} x_7 q_{\alpha,\beta} = 0 \right\} \subset \projspace^7_{\overline{K}}
        $$
        of $\mathcal{X} \to \Spec R$ is a smooth (3,3) complete intersection.
    \end{lemma}

    \begin{proof}
        Recall from \Cref{rem:CheckSmoothnessAfterSpecialization} that it suffices to check smoothness after some specialization. Note that
        \begin{align*}
            Y \quad &\text{specializes via } \ \beta \to \infty &&\text{to } \left\{c_2 = 0\right\} \subset \projspace^6, \\
            Z \quad &\text{specializes via } \ \alpha \to \infty &&\text{to } \left\{c_1 = q_1 = 0 \right\} \subset \projspace^7, \\
            W \quad &\text{specializes via } \ \beta \to \infty &&\text{to } \left\{c_2 = q_2 = 0\right\} \subset \projspace^6, \\
            X_{\overline{K}} \quad &\text{specializes via } \ t \to \infty \text{ and } \alpha \to \infty &&\text{to } \left\{c_1 = f_1 = 0\right\} \subset \projspace^7.
        \end{align*}
        The varieties on the right hand side are smooth by our choices in \Cref{def:polynomials}.
    \end{proof}
    Our current model $\mathcal{X} \to \Spec R$ is proper and flat. Moreover the irreducible components of the special fibre and their intersections are smooth. However the components of the special fibre of $\mathcal{X} \to \Spec R$ are not Cartier in $\mathcal{X}$, i.e. $\mathcal{X} \to \Spec R$ is not strictly semi-stable.
    \begin{lemma}\label{lem:SingularLocusOfCurlyX}
        The singular locus of the total space $\mathcal{X}$ is given by
        $$
            S := \left\{c_{\alpha,\beta,\gamma} = f_{\alpha,\beta} = q_{\alpha,\beta} = x_7 = t = 0 \right\} \subset \mathcal{X}.
        $$
        Furthermore, $S$ is smooth and $\mathcal{X}$ has ordinary quadratic singularities along $S$.
    \end{lemma}

    \begin{proof}
        First we show that $S$ is smooth. By definition, $S$ is isomorphic to the variety
        $$
            \left\{c_0 + \gamma(x_6 p_3 + c_3) + \beta c_2 + \alpha c_1 = q_0 + \beta q_2 + \alpha q_1 = x_6^3 + \beta f_2 + \alpha f_1 = x_7 = 0\right\} \subset \projspace^7_k.
        $$
        We note that $S$ specializes via $\alpha \to \infty$ to $\{c_1 = f_1 = q_1 = x_7 = 0\} \subset \projspace^7$ which is smooth by assumption in \Cref{def:polynomials}. Hence $S$ is smooth by \Cref{rem:CheckSmoothnessAfterSpecialization}.

        Next we check that $S$ is indeed the singular locus of $\mathcal{X}$. Recall that the singular locus $\Sing \mathcal{X}$ of $\mathcal{X}$ is given by the vanishing of the defining equations of $\mathcal{X}$ as well as all minors of the Jacobian. The Jacobian of $\mathcal{X}$ is given by
        \begin{equation}\label{eq:JacobianOfCurlyX}
            \begin{pmatrix}
                 \partial_0 c_{\alpha,\beta,\gamma} & \dots & \partial_6 c_{\alpha,\beta,\gamma} & \partial_7 c_{\alpha,\beta,\gamma} & 0 \\
                 t \partial_0 f_{\alpha,\beta} + x_7 \partial_0 q_{\alpha,\beta} & \dots & t \partial_6 f_{\alpha,\beta} + x_7 \partial_6 q_{\alpha,\beta} & t \partial_7 f_{\alpha,\beta} + q_{\alpha,\beta} + x_7 \partial_7 q_{\alpha,\beta} & f_{\alpha,\beta}
            \end{pmatrix}.
        \end{equation}
        Obviously, the variety $S$ is contained in $\Sing \mathcal{X}$ because the defining equation and the second row of \eqref{eq:JacobianOfCurlyX} vanish. We show the opposite inclusion: Since the geometric generic fibre $X_{\overline{K}}$ is smooth by \Cref{lem:ComponentsOfSpecialFibreAreSmooth}, the singular locus of $\mathcal{X}$ is contained in the special fibre. We further note that $f_{\alpha,\beta}$ vanishes at every point of the singular locus of $\mathcal{X}$, because $\{c_{\alpha,\beta,\gamma} = 0\} \subset \projspace^7_k$ is smooth (as it specializes via $\alpha \to \infty$ to the smooth variety $\{c_1 = 0\} \subset \projspace^7$). In particular, we see that the singular locus is contained in $$
            \left\{c_{\alpha,\beta,\gamma} = f_{\alpha,\beta} = t = x_7 q_{\alpha,\beta} = 0\right\} \subset \mathcal{X}.
        $$
        Hence it suffices to show under the assumption $c_{\alpha,\beta,\gamma} = f_{\alpha,\beta} = t = 0$ that $$
            x_7 = 0 \Longleftrightarrow q_{\alpha,\beta} = 0.
        $$
        We start by showing the implication left to right. Note that $\{c_{\alpha,\beta,\gamma} = x_7 = 0\} \subset \projspace^7_k$ is smooth by \Cref{rem:CheckSmoothnessAfterSpecialization} as it specializes via $\beta \to \infty$ to the smooth variety $\{c_2 = 0\}$. Thus, we conclude that $q_{\alpha,\beta} = 0$ as wanted because otherwise \eqref{eq:JacobianOfCurlyX} has rank $2$. For the implication right to left, we note that $$
            \{c_{\alpha,\beta,\gamma} = q_{\alpha,\beta} = 0\} \subset \projspace^7_k
        $$
        is smooth because it specializes via $\alpha \to \infty$ to the smooth variety $\{c_1 = q_1 = 0\} \subset \projspace^7$. Thus, $x_7$ has to be equal to $0$ as otherwise the Jacobian would have full rank. This shows $\Sing \mathcal{X} \subset S$ and thus $S = \Sing \mathcal{X}$.

        Lastly, we describe the type of the singularities of $\mathcal{X}$. Let $P \in S$ be any point, i.e. $P$ is a singular point of $\mathcal{X}$. The varieties $\{f_{\alpha,\beta} = 0 \}, \ \{ q_{\alpha,\beta}  = 0 \} \subset \projspace^7_k$ are smooth along $S$, because this holds after the specialization $\alpha \to \infty$ by construction in \Cref{def:polynomials}. Thus the tangent cone of $\{t f_{\alpha,\beta} + x_7 q_{\alpha,\beta} = 0\} \subset \projspace^7_R$ at $P$ is Zariski locally isomorphic to the tangent cone of the ordinary quadratic singularity $\{tx + yz = 0\}$. Moreover, the tangent cone of $\{t f_{\alpha,\beta} + x_7 q_{\alpha,\beta} = 0\} \subset \projspace^7_R$ at $P$ intersects the tangent space of $\{c_{\alpha,\beta,\gamma} = 0\} \subset \projspace^7_R$ at $P$ transversely because $\{c_{\alpha,\beta,\gamma} = 0\} \subset \projspace^7_k$ is smooth (as it specializes via $\beta \to \infty$ to the smooth cubic hypersurface $\{c_1 = 0\} \subset \projspace^7$). This concludes the proof of the lemma.
    \end{proof}

    We obtain a strictly semi-stable model by blowing up one irreducible component of the special fibre.

    \begin{lemma}\label{lem:CurlyX'}
        The blow-up $\mathcal{X}' := \blowup_{Y} \mathcal{X} \to \Spec R$ is strictly semi-stable with special fibre $\Tilde{Y} \cup Z$ where $\Tilde{Y} := \blowup_S Y$. Moreover the scheme-theoretic intersection $\Tilde{Y} \cap Z$ is isomorphic to $W$.
    \end{lemma}

    \begin{proof}
        The family $\mathcal{X}' \to \Spec R$ is proper and is flat by \cite[III. Proposition 9.7]{Har77}. Locally at a point of $S$, $\mathcal{X}$ has ordinary quadratic singularities (see \Cref{lem:SingularLocusOfCurlyX}) and a local computation shows that the special fibre of $\mathcal{X}'$ is given by $\Tilde{Y} \cup Z$, where $\Tilde{Y} = \blowup_S Y$. We sketch the computation for the convenience of the reader: Recall, $\mathcal{X} = \{tf + x_7 q = c = 0\} \subset \projspace^7_{k[[t]]}$ and $Y = \{t = x_7 = 0\} \subset \mathcal{X}$, where we suppressed the indices for simplicity. The blow-up $\blowup_Y \mathcal{X}$ is then given as follows:
        \begin{equation}\label{eq:localcharts}
            \begin{aligned}
                \underline{\text{in the }t\text{-chart:}} & \quad \{x_7 - t x_7' = f + x_7' q = c = 0\},  \\
                \underline{\text{in the }x_7\text{-chart:}} & \quad \{t - x_7 t' = t' f + q = c = 0\}.
            \end{aligned}
        \end{equation} The special fibre $\{t = 0\}$ consists of the two components $\{t = x_7 = 0\}$ and $\{t' = 0\}$. While the latter is isomorphic to $Z$, the former is isomorphic to $\blowup_S Y$ because $S = \{f = q = 0\} \subset Y$.

        Moreover, we find that the scheme-theoretic intersection $\Tilde{Y} \cap Z = \blowup_S W = W$, where the final equality holds because $S \subset W$ is a Cartier divisor. By \Cref{lem:ComponentsOfSpecialFibreAreSmooth} and \Cref{lem:SingularLocusOfCurlyX} $Y$, $Z$, $W$, and $S$ are smooth, so all components of the special fibre of $\mathcal{X}' \to \Spec R$ and their intersection are smooth. By construction $\Tilde{Y}$ is Cartier and $Z \subset \mathcal{X}'$ is also Cartier, because the special fibre is Cartier and reduced.
    \end{proof}

    The obstruction to rationality lies in $W$, see \cite[Theorem 7.2]{NO22}. Hence we need to blow-up $W$ to obtain a component in the special fibre which is stably birational to $W$ and can be seen by the obstruction map \eqref{eq:ObstructionMap}. In order to ensure that the model will remain strictly semi-stable, we perform a $2:1$ base change first, see also \cite[Lemma 5.5]{PS23}.

    \begin{lemma}\label{lem:TildeCurlyX}
        Let $\mathcal{X}'' := \mathcal{X}' \times_{\substack{R \to R \\ t \to t^2}} R$ be the $2{:}1$ base change of $\mathcal{X}'$. The blow-up
        \begin{equation}\label{eq:strictlysemistablemodel}
            \Tilde{\mathcal{X}} := \blowup_Z \mathcal{X}'' \longrightarrow \Spec R
        \end{equation}
        is a strictly semi-stable $R$-scheme with special fibre $\Tilde{X}_k = \Tilde{Y} \cup P_W \cup Z$, where $P_W$ is a $\projspace^1$-bundle over $W$ and $\Tilde{Y} = \blowup_S Y$ as in \Cref{lem:CurlyX'}. The intersections $\Tilde{Y} \cap P_W$ and $Z \cap P_W$ are disjoint sections of the bundle $P_W \to W$. The geometric generic fibre
        $$
            \Tilde{X}_{\overline{K}} = \left\{c_{\alpha,\beta,\gamma} = f_{\alpha,\beta} + t^{-2} x_7 q_{\alpha,\beta} = 0 \right\} \subset \projspace^7_{\overline{K}}
        $$
        is a smooth (3,3) complete intersection.
    \end{lemma}
    
    \begin{proof}
        Recall the local description of $\mathcal{X}'$ in \eqref{eq:localcharts}. From this, we see that the $2:1$-base change $\mathcal{X}''$ is given by $$
            \begin{aligned}
                \underline{\text{in the }t\text{-chart:}} & \quad \{x_7 - t^2 x_7' = f + x_7' q = c = 0\},  \\
                \underline{\text{in the }x_7\text{-chart:}} & \quad \{t^2 - x_7 t' = t' f + q = c = 0\}.
            \end{aligned}
        $$
        A similar argument as in the proof of \Cref{lem:SingularLocusOfCurlyX} shows that the singular locus of $\mathcal{X}''$ is given by $\{t' = t = x_7 = 0\} \subset \mathcal{X}'$, i.e. it is the singular locus $W$ of the special fibre. In particular we see that the $2:1$ base-change is regular away from $W$. We note that the component $Z \subset \mathcal{X}''$ is given by $\{t = t' = 0\}$ in the above local charts. Thus the blow-up $\blowup_Z \mathcal{X}''$ is given by $$
            \begin{aligned}
                \underline{\text{in the }t\text{-chart:}} & \quad \{x_7 - t^2 x_7' = f + x_7' q = c = 0\},  \\
                \underline{\text{in the }x_7-t\text{-chart:}} & \quad \{t'- t \tilde{t}' = t - x_7 \tilde{t}' = t \Tilde{t}' f + q = c = 0\}, \\
                \underline{\text{in the }x_7-t'\text{-chart:}} & \quad \{t - \Tilde{t}t' = t'\Tilde{t}^2 - x_7 = t' f + q = c = 0\}.
            \end{aligned}
        $$
        We see that the irreducible components of the special fibre are $\{t = \Tilde{t} = 0\}$, $\{\Tilde{t}' = 0\}$, and $\{x_7 = t' = 0\}$. Note that the first two components are the strict transform of $\Tilde{Y}$ and $Z$, respectively. We recognize the last component as a smooth conic bundle $P_W$ over $W$ which admits a section, e.g. by $Z \cap P_W = \{\Tilde{t}' = 0\} \subset P_W$. Hence $P_W$ is a $\projspace^1$-bundle over $W$ as claimed in the statement.
        
        Since the singularities of $\mathcal{X}''$ are ordinary quadratic singularities, the blow-up resolves them. Hence, $\Tilde{\mathcal{X}}$ is regular and in particular $\Tilde{\mathcal{X}} \to \Spec R$ is strictly semi-stable. The smoothness of the geometric generic fibre follows from \Cref{lem:CurlyX'}. 
    \end{proof}

    \subsection{Specialization}

    In the last section we constructed a strictly semi-stable family $\mathcal{X} \to \Spec R$. We aim to show that the obstruction map
    $$
        \Phi_{\Tilde{\mathcal{X}}} \colon \CH_1(\Tilde{X}_k) \longrightarrow \Ker \left(\CH_0(\Tilde{Y}) \oplus \CH_0(P_W) \oplus \CH_0(Z) \overset{\deg}{\longrightarrow} \Z \right)
    $$
    is not universally surjective modulo $2$. Specifically we try to understand the map
    \begin{equation}\label{eq:PhiTildeCurlyXMod2}
        \Phi_{\Tilde{\mathcal{X}},P_W} \colon \CH_1(\Tilde{X}_k) \longrightarrow \CH_0(P_W) \mod 2.
    \end{equation}
    From the construction of the strictly semi-stable model $\Tilde{\mathcal{X}} \to \Spec R$ in the \Cref{sec:sssfamily}, we can make a couple of observations in order to better understand the image of the map. Clearly we have a surjection,
    $$
        \CH_1(\Tilde{Y}) \oplus \CH_1(P_W) \oplus \CH_1(Z) \relbar\joinrel\twoheadrightarrow \CH_1(\Tilde{X}_k),
    $$
    The map is given by push-forwards of the corresponding inclusions of varieties. We consider first the contribution of $\CH_1(P_W)$. Since $P_W \to W$ is a $\projspace^1$-bundle by \Cref{lem:TildeCurlyX} and $W$ is smooth by \Cref{lem:ComponentsOfSpecialFibreAreSmooth}, there is an isomorphism \begin{equation}\label{eq:ChowFormulaForP1Bundle}
        \CH_0(W) \oplus \CH_1(W) \longrightarrow \CH_1(P_W).
    \end{equation}
    where the map on the first factor is the pull-back along the flat morphism $P_W \to W$ and the map on the second factor is the pushforward via a section $W \to P_W$ of the $\projspace^1$-bundle $P_W$. Since $P_W \cap Z$ is a section of $P_W \to W$, we find that the contribution of $\CH_1(W)$ in the obstruction map \eqref{eq:PhiTildeCurlyXMod2} is contained in the image of $\CH_1(Z)$ (or $\CH_1(\Tilde{Y})$). Moreover, the concrete description of the obstruction map \eqref{eq:ExplicitDescriptionOfObstructionMap} implies that the contribution of $\CH_0(W)$ vanishes in \eqref{eq:PhiTildeCurlyXMod2}. Indeed, a closed point $w \in W$ is mapped under the isomorphism \eqref{eq:ChowFormulaForP1Bundle} to the line $F_w$ over that point. By \eqref{eq:ExplicitDescriptionOfObstructionMap} the line $F_w$ is mapped via \eqref{eq:PhiTildeCurlyXMod2} to
    $$
        - [F_w] \cdot [\Tilde{Y}] - [F_w] \cdot [Z] = - 2 [z].
    $$
    where $z \in F_w$ is any point on the line $F_w \subset P_W$ and we used that $\Tilde{Y} \cap P_W$ and $Z \cap P_W$ are sections of the $\projspace^1$-bundle $P_W \to W$. Thus, the image of $\Phi_{\Tilde{\mathcal{X}},P_W}$ modulo $2$ is contained in
    $$
        \Ima \left(\CH_1(\Tilde{Y}) \oplus \CH_1(Z) \longrightarrow \CH_0(P_W)\right) \mod 2.
    $$
    Next we take a look at $\CH_1(\Tilde{Y})$. Recall that $\Tilde{Y} = \blowup_S Y$ where $S \subset \mathcal{X}$ is the singular locus of $\mathcal{X}$, see \Cref{lem:CurlyX'}. The blow-up formula for Chow groups (see e.g. \cite[Theorem 3.3 and Proposition 6.7 (e)]{Ful98}) yields a canonical isomorphism
    $$
        \CH_1(Y) \oplus \CH_0(S) \cong \CH_1(\Tilde{Y}).
    $$
    Thus we conclude that the image of $\Phi_{\Tilde{\mathcal{X}},P_W}$ modulo $2$ is contained in
    $$
        \Ima \left(\CH_1(Y) \oplus \CH_0(S) \oplus \CH_1(Z) \longrightarrow \CH_0(P_W)\right) \mod 2.
    $$
    These observations also hold after a field extension $L/k$, which shows the following.
    \begin{lemma}\label{rem:ReductionOfObstructionMapBeforeSpecialization}
        For any field extension $L/k$, the base change to $L$ of $\Phi_{\Tilde{\mathcal{X}},P_W}$ modulo $2$ has image contained in \begin{equation}\label{eq:ReductionOfObstructionMapBeforeSpecialization}
            \Ima \left(\CH_1(Y \times_k L) \oplus \CH_0(S \times_k L) \oplus \CH_1(Z \times_k L) \longrightarrow \CH_0(P_W \times_k L)\right) \mod 2.
        \end{equation}
    \end{lemma}

    The Chow groups of the domain are still hard to describe. The transcendental parameters $\alpha,\beta,$ and $\gamma$ allow us to degenerate the varieties further. Together with Fulton's specialization map this will enable us to describe the Chow groups and understand \eqref{eq:ReductionOfObstructionMapBeforeSpecialization}.

    \begin{lemma}[Fulton's specialization map; {\cite[Lemma 5.7]{PS23}}] \label{lem:FultonSpecialization} 
        Let $B$ be a discrete valuation ring with fraction field $F$ and residue field $L$. Let $p \colon \mathcal{X} \to \Spec B$ and $q \colon \mathcal{Y} \to \Spec B$ be proper, flat $B$-schemes with connected fibres. Denote by $X_\eta,Y_\eta$ and $X_0,Y_0$ the generic and the special fibres of $p, \ q$ respectively. Assume $Y_0$ is integral, i.e. $A = \mathcal{O}_{\mathcal{Y},Y_0}$ is a discrete valuation ring, and consider the flat proper $A$-scheme $\mathcal{X}_A \to \Spec A$, given by base change of $p$. Then Fulton's specialization map induces a specialization map
        $$
            \specialization \colon \CH_i(X_\eta \times_F \overline{F}(Y_\eta)) \longrightarrow \CH_i(X_0 \times_L \overline{L}(Y_0)),
        $$
        where $\overline{F}$ and $\overline{L}$ denote the algebraic closures of $F$ and $L$, respectively, such that the following holds:
        \begin{enumerate}[label=(\arabic*)]
            \item $\specialization$ commutes with pushforwards along proper maps and pullbacks along regular embeddings;
            \item If $\mathcal{X} = \mathcal{Y}$, then $\specialization(\delta_{X_\eta}) = \delta_{X_0}$, where $\delta_{X_\eta} \in \CH_0(X_\eta \times_F \overline{F}(X_\eta))$ and $\delta_{X_0} \in CH_0(X_0 \times_L \overline{L}(X_0))$ denote the diagonal points.
        \end{enumerate}
    \end{lemma}

    The first item of the lemma ensures that the specialization map $\specialization$ commutes with the obstruction map $\Phi_{\Tilde{\mathcal{X}}}$, i.e. to understand specializations of $\Phi_{\Tilde{\mathcal{X}}}$ it suffices to understand the specializations of the involved varieties. To distinguish the varieties from their specializations we denote by subscripts the transcendental parameters (i.e. $\alpha,\beta,\gamma$) on which the variety depends, e.g. $Z = Z_{\alpha,\beta}$. We omit a parameter after specializing it to zero and denote the scheme obtained after specializing all transcendental parameters $\alpha,\ \beta,$ and $\gamma$ to $0$ with subscript $0$.
    
    In the remainder of this section we prove the following result.
    \begin{proposition}\label{prop:KeyReduction}
        Let $\Phi_{k(P_W)}$ denote the obstruction map $\Phi_{\Tilde{\mathcal{X}},P_W}$ modulo $2$ where every scheme is base-changed to the function field $k(P_W)$ of $P_W$, i.e.
        $$
            \Phi_{k(P_W)} \colon \CH_1(\Tilde{X}_{k} \times_k k(P_W)) \longrightarrow \CH_0(P_W \times_k k(P_W)) \mod 2.
        $$
        Then the image of $\specialization_\gamma \circ \specialization_\beta \circ \specialization_\alpha \circ \Phi_{k(P_W)}$ is contained in the image of $$
            \CH_0(P_{W_0})/2 \longrightarrow \CH_0(P_{W_0} \times_{k_0} k_0(P_{W_0}))/2
        $$
        where $\specialization_i$ is the specialization obtained by sending $i \to 0$. Moreover, the diagonal point $\delta_{P_W} \in \CH_0(P_W \times_k k(P_W))$ is sent to the diagonal point $\delta_{P_{W_0}} \in \CH_0(P_{W_0} \times_{k_0} k_0(P_{W_0}))$ under these specializations.
    \end{proposition}

    We sketch the proof first: By \Cref{rem:ReductionOfObstructionMapBeforeSpecialization} the image of $\Phi_{k(P_W)}$ is contained in $$
        \Ima \left(\CH_1(Y \times_k k(P_W)) \oplus \CH_0(S \times_k k(P_W)) \oplus \CH_1(Z \times_k k(P_W)) \longrightarrow \CH_0(P_W \times_k k(P_W))\right) \mod 2.
    $$
    The specialization $\alpha,\beta,\gamma \to 0$ enables us to simplify the Chow groups as follows: We first specialize $Z$ to a singular variety which is birational to $Y$. This allows us to write $\CH_1(Z)$ as $\CH_1(Y)$ and some $\CH_0$. In the second step $Y$ becomes rational, i.e. $\CH_1(Y)$ is equal to $\CH_1(\projspace^5) \cong \Z$ plus some $\CH_0$. Lastly the remaining schemes, i.e. $S$ and the two introduced in the previous two steps, specialize to some schemes with universally trivial $\CH_0$. The diagram below visualizes this strategy in an informal way:
    \begin{equation}
        \label{eq:strategydiagram}
        \begin{tikzcd}[decoration={snake,amplitude = 1pt,segment length=5pt},
        column sep=large, row sep=0pt,
        /tikz/column 1/.append style={anchor=base east},
        /tikz/column 2/.append style={anchor=base west}
        ]
            Z \arrow[r,decorate,"\alpha \to 0"] & Y \ + \text{ something}, \\
            Y \arrow[r,decorate,"\beta \to 0"] & \projspace^5 \ + \text{ something}, \\
            S \  + \text{ something} \arrow[r,decorate,"\gamma \to 0"] & \text{something with universally trivial } \CH_0.
        \end{tikzcd}
    \end{equation}

    The following lemma allows us to make the above described simplification. Before stating the lemma we describe quickly the geometric picture. Consider a variety $Z \subset \projspace^n$ which is given as the scheme-theoretic intersection of degree $d$ hypersurface $H$ with a cone over a smooth variety $Y \subset \projspace^{n-1}$ with a $k$-rational point $Q$ as vertex such that the hypersurface $H$ intersects $Q$ with multiplicity $d-1$. Thus the projection from $Q$ yields a birational map $Z \dashrightarrow Y$. By resolving the birational map we can describe $\CH_1(Z)$ in terms of $\CH_1(Y)$ and $\CH_0$ of the exceptional locus of the map.

    \begin{lemma}\label{lem:SimplificationCH1}
        Let $Y := \{F_1 = \dots = F_r = 0\} \subset \projspace^n_\kappa$ be a smooth variety over a field $\kappa$ where $F_1,\dots,F_r \in \kappa[x_0,\dots,x_n]$ are homogeneous polynomials and $r \geq 0$. Let $$
            Z := \{F_1 = \dots = F_r = g_1 x_{n+1} + g_0 = 0\} \subset \projspace^{n+1}_\kappa
        $$
        where $g_i \in k[x_0,\dots,x_n]$ are homogeneous polynomials of degree $d - i$. Assume further that $$
            W := \{F_1 = \dots = F_r = g_1 = g_0 = 0\} \subset \projspace^n_\kappa
        $$
        is smooth. Then for any field extension $\kappa' / \kappa$ there is a surjective homomorphism
        $$
            \CH_0(W \times_\kappa \kappa') \oplus \CH_1(Y \times_\kappa \kappa') \relbar\joinrel\twoheadrightarrow \CH_1(Z \times_\kappa \kappa').
        $$
    \end{lemma}

    \begin{remark}
        Our assumption on $W$ and $Y$ implies that $Z$ is smooth away from $Q$.
    \end{remark}

    \begin{proof}
        We note first that $Z$ is the scheme-theoretic intersection of the cone $C_Y$ over $Y \subset \projspace^n_\kappa \cong \{x_{n+1} = 0\} \subset \projspace^{n+1}_\kappa$ with vertex $Q = [0:\dots:0:1] \in \projspace^{n+1}_\kappa$ and the degree $d$ hypersurfaces $H := \{g_1 x_{n+1} + g_0 = 0\} \subset \projspace^{n+1}_\kappa$. Since $C_Y$ is a cone with vertex $Q$, the projection from $Q$ induces a rational map
        $$
            \varphi \colon C_Y \dashrightarrow Y.
        $$
        As $H$ has multiplicity $d-1$ at $Q$, the restriction of $\varphi$ to $Z = C_Y \cap H$ is a birational map $\restr{\varphi}_Z \colon Z \dashrightarrow Y$. Indeed the fibre of $\varphi$ over some $\kappa$-rational point $P \in Y$, different from $Q$, is the unique line through $P$ and $Q$. Since $H$ has multiplicity $d-1$ at $Q$, $H$ intersects this line in a unique other point which is mapped to $P$ under $\varphi$. An explicit computation in affine charts yields that the map $\restr{\varphi}_Z$ is resolved by the isomorphism
        $$
            \blowup_Q Z \overset{\cong}{\longrightarrow} \blowup_W Y.
        $$
        We sketch the computation for the convenience of the reader: Recall, $$
            Q = \{x_0 = \dots = x_n = 0\} \in Z \subset \projspace^{n+1},
        $$
        i.e. the $x_i$-chart of the blow-up $\blowup_Q Z$ is given by
            $$
                \{x_j - X_j x_i = F_1(X) = \dots = F_r(X) = g_1(X) x_{n+1} + x_i g_0(X) = 0, \quad \text{for } j \neq i,n+1\}.
            $$
            On the other hand, we consider the standard affine charts $U_i \subset \projspace^n$ with affine coordinates $Y_j = \frac{y_j}{y_i}$ for $j \neq i$. The blow-up of $U_i$ along $W \cap U_i$ is given by
            $$
                \{F_1(Y) = \dots = F_r(Y) = g_0(Y) T - S g_1(Y) = 0\}.
            $$
            From this description, we immediately see that $\blowup_Q Z$ and $\blowup_W Y$ are locally isomorphic and thus isomorphic, as they are birational. Note moreover, the isomorphism is induced by the birational map $\restr{\varphi}_Z$.
            
        Hence, there is an isomorphism on the level of Chow groups $\CH_1(\blowup_W Y) \cong \CH_1(\blowup_Q Z)$. Since $W$ and $Y$ are smooth by assumption, the blow-up formula for Chow groups (see e.g. \cite[Theorem 3.3 and Proposition 6.7 (e)]{Ful98}) yields
        $$
            \CH_0(W) \oplus \CH_1(Y) \cong \CH_1(\blowup_W Y) \cong \CH_1(\blowup_Q Z).
        $$
        Moreover, the natural pushforward $$
            \pi_\ast \colon \CH_1(\blowup_Q Z) \longrightarrow \CH_1(Z)
        $$
        of the proper morphism $\pi \colon \blowup_Q Z \to Z$ is surjective, because $Q$ is a point.

        Since blow-ups commute with extension of the base field, the above construction also works after any base extension, i.e. we obtain a surjective homomorphism
        $$
            \CH_0(W \times_\kappa \kappa') \oplus \CH_1(Y \times_\kappa \kappa') \relbar\joinrel\twoheadrightarrow \CH_1(Z \times_\kappa \kappa'),
        $$
        concluding the proof.
    \end{proof}

    \begin{remark}
        A typical use of \Cref{lem:SimplificationCH1} involves $Y$ to simply be some projective space, so $\CH_1(Y) \simeq \Z$. In this case we have an alternative viewpoint to \Cref{lem:SimplificationCH1} through the localization exact sequence of Chow groups. 
        With notation as in the lemma, let $C_W$ be the cone over $W$, and define $U \coloneqq Y \setminus W$. Since the birational map $\varphi$ is an isomorphism away from $C_W$ and $W$ respectively, we have short exact sequences, valid over any field extension:
        \begin{align*}
            \CH_1(C_W) \longrightarrow \CH_1(Z) \longrightarrow \CH_1(U) \longrightarrow 0 \\
            \CH_1(W) \longrightarrow \CH_1(Y) \longrightarrow \CH_1(U) \longrightarrow 0
        \end{align*}
        Assume $Y=\projspace^n$ and $W$ contains a line. Then it follows from the bottom sequence that $\CH_0(U)$ is trivial, hence $\CH_1(C_W) \to \CH_1(Z)$ is surjective.  Since $C_W$ is a cone over $W$, the complement of the vertex is an affine bundle over $W$. It furthermore follows from the formula for $\CH_1$ of an affine bundle (\cite[Proposition 1.9]{Ful98}) that there is an isomorphism $\CH_1(C_W) \simeq \CH_0(W)$, valid over any field extension. In total, this approach gives a surjection $\CH_0(W) \to \CH_1(Z)$, valid over any field extension, just as \Cref{lem:SimplificationCH1}.        
    \end{remark}

    \begin{proof}[Proof of \Cref{prop:KeyReduction}]
        We aim to prove that the image of the base-changed obstruction map
        $$
            \Phi_{k(P_W)} \colon \CH_1\left(\left(X_{k}\right)_{\alpha,\beta,\gamma} \times_k k\left(P_{W_{\alpha,\beta,\gamma}}\right)\right) \longrightarrow \CH_0\left(P_{W_{\alpha,\beta,\gamma}} \times_k k\left(P_{W_{\alpha,\beta,\gamma}}\right)\right) \mod 2
        $$
        is contained in the image of \begin{equation}\label{eq:ImageOfPhiAfterAllSpecializations}
            \CH_0\left(P_{W_0}\right) \longrightarrow \CH_0\left(P_{W_0} \times_{k_0} k_0\left(P_{W_0}\right)\right) \mod 2,
        \end{equation}
        after specializing $\alpha, \ \beta, $ and $\gamma$ to $0$. Recall that we noticed in \Cref{rem:ReductionOfObstructionMapBeforeSpecialization} that the image of $\Phi_{k(P_W)}$ is contained in the image of the homomorphism
        \begin{equation}\label{eq:KeyReductionStep0}
            \CH_1(Y_{\alpha,\beta,\gamma} \times L_0)) \oplus \CH_0(S_{\alpha,\beta,\gamma} \times L_0) \oplus \CH_1(Z_{\alpha,\beta,\gamma} \times L_0) \longrightarrow \CH_0(P_{W_{\alpha,\beta,\gamma}} \times L_0) \mod 2,
        \end{equation}
        where $L_0 := k(P_{W_{\alpha,\beta,\gamma}})$ is the function field of the $\projspace^1$-bundle $P_{W_{\alpha,\beta,\gamma}}$. Thus we need to show that the image of \eqref{eq:KeyReductionStep0} is contained in \eqref{eq:ImageOfPhiAfterAllSpecializations} after applying $\specialization_\gamma \circ \specialization_\beta \circ \specialization_\alpha$. This is done by analyzing each specialization in the following three steps. In order to apply Fulton's specialization map, i.e. \Cref{lem:FultonSpecialization}, in each step, we need to check that $P_W$ remains integral after specialization. Since $P_W$ is a $\projspace^1$-bundle over $W$, it suffices to check that $W$ remains integral after each specialization.

        \textit{Step 1.} We specialize $\alpha \to 0$ to control $\CH_1(Z)$. Note that $c_{\alpha,\beta,\gamma}$ and $q_{\alpha,\beta,\gamma}$ specialize to
        \begin{align}
        \begin{split}\label{eq:SpecializedEquationsInStep1}
            c_{\beta,\gamma} &= c_0 + \gamma (x_6 p_3 + c_3) + \beta c_2 \in \overline{k_0(\beta,\gamma)}[x_0,\dots,x_6], \\
            q_{\beta,\gamma} &= (\beta x_3) \cdot x_7 + q_0 + \beta q_2 \in \overline{k_0(\beta,\gamma)}[x_0,\dots,x_7],
        \end{split}
        \end{align}
        respectively. Hence we find that $W_{\beta,\gamma} = \{c_{\beta,\gamma} = q_{\beta,\gamma} = x_7 = 0\} \subset \projspace^7$ is smooth and in particular integral, because it specializes to the smooth (2,3) fourfold $\{c_2 = q_2 = 0\} \subset \projspace^6$. Thus we can apply \Cref{lem:FultonSpecialization}. Moreover, we see from \eqref{eq:SpecializedEquationsInStep1} that $$
            Z_{\beta,\gamma} = \left\{(\beta x_3) \cdot x_7 + \left(q_0 + \beta q_2\right) = 0\right\} \cap C_{Y_{\beta,\gamma}} \subset \projspace^7,
        $$
        where $C_{Y_{\beta,\gamma}} \subset \projspace^7$ denotes the cone over $Y_{\beta,\gamma} \subset \projspace^6$ with vertex $[0:\dots:0:1]$, i.e. $Z_{\beta,\gamma}$ is of the form from \Cref{lem:SimplificationCH1}. It is immediate to check that $Y_{\beta,\gamma}$ and $V_{\beta,\gamma} := \{x_3 = q_0 + \beta q_2 = 0\} \subset \projspace^6$ are smooth. Thus, \Cref{lem:SimplificationCH1} yields a surjection
        $$
            \CH_1(Y_{\beta,\gamma} \times_\kappa L_1) \oplus \CH_0(V_{\beta,\gamma} \times_\kappa L_1) \relbar\joinrel\twoheadrightarrow \CH_1(Z_{\beta,\gamma} \times_\kappa L_1),
        $$
        where $\kappa := \overline{k_0(\beta,\gamma)}$ and $L_1 := \kappa(P_{W_{\beta,\gamma}})$. Hence, by applying \Cref{lem:FultonSpecialization} we find that the image of $\specialization_\alpha$ applied to \eqref{eq:KeyReductionStep0} is contained in the image of
        \begin{equation}\label{eq:KeyReductionStep1}
            \CH_1(Y_{\beta,\gamma} \times_\kappa L_1) \oplus \CH_0(S_{\beta,\gamma} \times_\kappa L_1) \oplus \CH_0(V_{\beta,\gamma} \times_\kappa L_1) \longrightarrow \CH_0(P_{W_{\beta,\gamma}} \times_\kappa L_1) \mod 2.
        \end{equation}

        \textit{Step 2.} $Y$ specializes to a cubic hypersurface with an ordinary double point via $\beta \to 0$, so $Y$ becomes rational. Note that after applying $\beta \to 0$, $c_{\beta,\gamma}$ and $q_{\beta,\gamma}$ become 
        \begin{align}
        \begin{split}\label{eq:SpecializedEquationsInStep2}
            c_{\gamma} &= (x_2^2 - 2x_3^2 + \gamma p_3) \cdot x_6 + x_0^2 x_5 + x_1^2 x_4 + x_3(x_3^2 + x_4^2 + x_5^2 - 2x_3(x_4+ x_5)) + \gamma c_3, \\
            q_{\gamma} &= q_0 = x_3x_6 - x_4 x_5,
        \end{split}
        \end{align}
        in $\overline{k_0(\gamma)}[x_0,\dots,x_6]$. We see that $W_{\gamma} = \{c_\gamma = q_\gamma = 0\} \subset \projspace^6$ is integral, i.e. we can apply \Cref{lem:FultonSpecialization}. Furthermore, we see from \eqref{eq:SpecializedEquationsInStep2} that $$
            Y_{\gamma} = \{c_{\gamma} = 0\} \subset \projspace^6
        $$
        has an ordinary double point singularity at $[0:\dots:0:1]$, i.e. $Y_\gamma$ is rational. We aim to apply \Cref{lem:SimplificationCH1} again. Note that
        $$
            Y_\gamma = \left\{(\gamma p_3) \cdot x_6 + \left(c_0 + \gamma c_3\right) = 0\right\} \cap \projspace^6 = \left\{(\gamma p_3) \cdot x_6 + \left(c_0 + \gamma c_3\right) = 0\right\} \cap C_{\projspace^5} \subset \projspace^6
        $$
        where we view $C_{\projspace^5} = \projspace^6$ as the cone over the hyperplane $\{x_6 = 0\}$ with vertex $[0:\dots:0:1]$. It is immediate that $\projspace^5$ and $U_{\gamma} := \{p_3 = c_0 + \gamma c_3 = 0\} \subset \projspace^5$ are smooth. Thus, \Cref{lem:SimplificationCH1} yields a surjection
        $$
            \Z \oplus \CH_0(U_\gamma \times_\kappa L_2) \cong \CH_1(\projspace^5_{L_2}) \oplus \CH_0(U_\gamma \times_\kappa L_2) \relbar\joinrel\twoheadrightarrow \CH_1(Y_\gamma \times_\kappa L_2),
        $$
        where $\kappa = \overline{k_0(\gamma)}$ and $L_2 = \kappa(P_{W_\gamma})$. Moreover, we used that $\CH_1$ of projective space is generated by a line, i.e. isomorphic to $\Z$. Since the line can be choosen to be defined over $\kappa$, we find that the image of the map $\CH_1(\projspace_{L_2}) \to \CH_0(P_{W_\gamma} \times_\kappa L_2)$ is contained in the image of the map $\CH_0(P_{W_\gamma}) \to \CH_0(P_{W_\gamma} \times_\kappa L_2)$. Hence, we find that the image of $\specialization_\beta \circ \specialization_\alpha$ applied to \eqref{eq:KeyReductionStep0} (or the image of $\specialization_\beta$ applied to \eqref{eq:KeyReductionStep1}) is contained in the image of
        \begin{small}
        \begin{equation}\label{eq:KeyReductionStep2}
            \CH_0(P_{W_\gamma}) \oplus \CH_0(U_\gamma \times_\kappa L_2) \oplus \CH_0(V_\gamma \times_\kappa L_2) \oplus \CH_0(S_\gamma \times_\kappa L_2) \longrightarrow \CH_0(P_{W_\gamma} \times_\kappa L_2) \mod 2.
        \end{equation}
        \end{small}

        \textit{Step 3.} After specializing $\gamma \to 0$, $U$, $V$, and $S$ become universally $\CH_0$-trivial. 
        
        Note that $W_\gamma$ specializes to
        $$
            W_0 = \{c_0 = q_0 = 0\} \subset \projspace^6
        $$
        which is integral, i.e. we can apply \Cref{lem:FultonSpecialization}. The schemes $U_\gamma, \ V_\gamma,$ and $S_\gamma$ specialize to
        \begin{align*}
            U_0 &= \{x_2^2 - 2x_3^2 = x_0^2 x_5 + x_1^2 x_4 + x_3(x_3^2 + x_4^2 + x_5^2 - 2x_3(x_4+ x_5)) = 0 \} \subset \projspace^5_{k_0},\\
            V_0 &= \{x_3 = x_4x_5 = c_0 = 0\} \subset \projspace^6_{k_0},\\
            S_0 &= \{x_6^3 = x_3x_6 - x_4x_5 = c_0 = 0 \} \subset \projspace^6_{k_0}.
        \end{align*}
        We claim that they all have universally trivial $\CH_0$. Then for each of these schemes the image of $$
            \CH_0(\cdot \times_{k_0} k_0(P_{W_0})) \to \CH_0(P_{W_0} \times_{k_0} k_0(P_{W_0}))
        $$ is contained in the image of the homomorphism $$
            \CH_0(P_{W_0}) \to \CH_0(P_{W_0} \times_{k_0} k_0(P_{W_0})).
        $$
        In particular we find that the image of $\specialization_\gamma \circ \specialization_\beta \circ \specialization_\alpha \circ \Phi_{k(P_W)}$ is contained \eqref{eq:ImageOfPhiAfterAllSpecializations}. The claim follows immediately from \Cref{lem:CH_0(U)} and \Cref{lem:CH_0(V)} below and thus the proposition holds.
    \end{proof}
    
    It remains to check that the $\CH_0$-groups of $S_0$, $U_0$ and $V_0$ are universally trivial. To do this we will apply the following two results from \cite{CTPCyclic}.
    \begin{lemma}[{\cite[Lemma 2.2]{CTPCyclic}}]
	\label{lem:RationalCH0Trivial}
	Let $k$ be an algebraically closed field and $X$ an integral projective $k$-rational variety. If $X$ is smooth on the complement of a finite number of closed points, then $\CH_0(X)$ is universally trivial.
\end{lemma}
\begin{lemma}[{\cite[Lemma 2.4]{CTPCyclic}}]
	\label{lem:CH0TrivialComponents}
	Let $X$ be a projective, reduced, geometrically connected scheme over a field $k$ and $X = \bigcup_{i=1}^N X_i$ its decomposition into irreducible components. Assume that
	\begin{enumerate}[label = \arabic*)]
		\item each $X_i$ is geometrically irreducible and every $\CH_0(X_i)$ is universally trivial,
		\item each intersection $X_i \cap X_j$ is either empty or contains a $0$-cycle of degree $1$.
	\end{enumerate}
	Then $\CH_0(X)$ is universally trivial.
\end{lemma}
Another useful observation is the following lemma, proving that a variety that is a cone with a rational point as its vertex is universally $\CH_0$-trivial.
\begin{lemma}
    \label{lem:ConesTrivial}
    Assume that the projective variety $X \subset \projspace^n$, defined over a field $k$, is a cone with a $k$-rational point $P$ as vertex, then $\CH_0(X)$ is universally trivial.
\end{lemma}
\begin{proof}
    For any field extension $K/k$ $X_K$ is a cone with a $K$-rational point as vertex. So it suffices to prove that for a cone $X$ defined over a field $k$, not necessarily algebraically closed, $CH_0(X) \simeq \Z$. To this end, let $Q \in X$ be any closed point in $X$. Assume that the residue field of $Q$ has degree $r$ over $k$. It suffices to prove that $Q$ is rationally equivalent to $rP$, where $P$ is the vertex of the cone.

    To see this, we consider the base change to the algebraic closure $\overline{k}$ of $k$. Here the inverse image of $Q$ is a union of $r$ closed points $Q_1, \dots, Q_r$, and since the base change remains a cone, each of these points can be connected to $P$ via a line $L_i$. So the points are rationally equivalent, meaning that 
    \[rP_{\Bar{k}} - \sum_{i=1}^r Q_i = \divisor(f)\]
    for some rational function $f$ on the union. Both the union of all the lines $\cup_{i=1}^r L_i$ and $f$ are invariant under the action of the Galois group, hence descend to $X$, and prove the rational equivalence of $rP$ and $Q$.
\end{proof}

Using these results, we can prove the following lemmas.
\begin{lemma}
    \label{lem:CH_0(U)}
    $\CH_0(U_0)$ is universally trivial.
\end{lemma}
\begin{proof}
    Recall that $U_0$ is defined by the equations:
    \begin{equation*}
  x_0^2x_5 + x_1^2x_4 + x_3(x_5^2 + x_4^2 + x_3^2 - 2x_3(x_5 + x_4))=x_2^2-2x_3^2=0
\end{equation*}
Since the ground field is algebraically closed, the quadric $x_2^2 - 2x_3^2 = 0$ is the union of the two hyperplanes. Hence $U$ is the union of two cubic threefolds, which are both isomorphic to the cubic threefolds defined in $\projspace^5$ by
$x_2 = x_0^2x_5 + x_1^2x_4 + x_3(x_5^2 + x_4^2 + x_3^2 - 2x_3(x_5 + x_4))=0$. From the partial derivatives with respect to $x_0$ and $x_1$ we recognize that any singular point must satisfy one of the four possibilities $x_4=x_5=0$, $x_0=x_4=0$, $x_1=x_5=0$ or $x_0=x_1=0$. It is straightforward to check that there are no singular points satisfying the first condition, and two singular points for each of the three remaining conditions.

All the singular points are ordinary double points, so this cubic threefold is rational. Hence it follows from \Cref{lem:RationalCH0Trivial} that each component of the union has universally trivial $\CH_0$ group. We can therefore conclude that $\CH_0(U_0)$ is universally trivial by \Cref{lem:CH0TrivialComponents}.
\end{proof}
\begin{lemma}
    \label{lem:CH_0(V)}
    $\CH_0(V_0)$ and $\CH_0(S_0)$ are universally trivial.
\end{lemma}
\begin{proof}
    Recall that $V_0$ is defined by the equations:
    \[x_3 = x_4x_5 = x_0^2x_5 + x_1^2x_4 + x_6x_2^2 = 0.\] 
    
    We recognize the scheme defined by these equations as the union of two cubic threefolds. Each cubic threefold is a cone over a cubic surface, and as such universally $\CH_0$ trivial by \Cref{lem:ConesTrivial}. The conclusion therefore follows from \Cref{lem:CH0TrivialComponents}.

    Since Chow groups only depend on the underlying reduced scheme, for $S_0$ we consider the following equations, which define $S_0^{red}$:
    \[x_6=x_4x_5 = x_0^2x_5 + x_1^2x_4 + x_3(x_3^2 + x_4^2 + x_5^2 - 2x_3(x_4+x_5)) = 0.\]
    By arguing as in the case of $V_0$ we see that also $S_0^{red}$ is universally $\CH_0$ trivial.
\end{proof}

\subsection{Proving that $\Phi$ is not surjective}
Recall from \Cref{prop:KeyReduction} that the image of $\Phi \colon \CH_1(X_k \times k(P_W)) \to \CH_0(P_{W} \times k(P_W))$ modulo $2$ is contained in the image of the base change map
\begin{equation}
\label{eq:PWBaseChangeMap}
\CH_0(P_{W_0})/2 \longrightarrow \CH_0(P_{W_0, k(P_{W_0})})/2.
\end{equation}
To simplify notation in the following part, we will drop the subscript $0$, writing $\mathcal{W}$ and $k$ for $W_0$ and $k_0$ respectively.

We can therefore prove that $\Phi$ is not surjective by proving the following:
\begin{proposition}
\label{lem:MerkurjevObstruction}
    The class 
    \begin{equation}
        \label{eq:diagonalclass}
        \delta_{P_\mathcal{W}} - z_{k(P_\mathcal{W} )}
    \end{equation}
    is not contained in the image of \eqref{eq:PWBaseChangeMap}. 
\end{proposition}
The proof is based on the Merkurjev pairing, introduced in \cite[Section 2.4]{Mer08}, and the proof is essentially the same as the proof of \cite[Proposition 3.12]{Ska23}, which in turn is based on methods of Schreieder in \cite{Sch19b}. See also \cite[Lemma 5.13]{PS23}.

Recall that on a smooth variety $X$ over a field $K$ of characteristic different from $2$, not necessarily algebraically closed, the Merkurjev pairing gives a bilinear pairing
\[\CH_0(X) \times H_{nr}^i(K(X)/K, \Z/2) \to H^i(K,\Z/2).\]
For an overview of the pairing and its application to rationality problems, see \cite[Section 5]{Sch21}.

\begin{proof}[Proof of \Cref{lem:MerkurjevObstruction}]
    Since $P_\mathcal{W} $ is a projective bundle over $\mathcal{W} $, we have an isomorphism $\CH_0(P_\mathcal{W} ) \simeq \CH_0(\mathcal{W} )$. Since this isomorphism also holds after extending the field to $k(\mathcal{W} )$, and Chow groups do not change under purely transcendental field extensions such as $k(P_\mathcal{W} )/k(\mathcal{W} )$, we also have an isomorphism $\CH_0(P_{\mathcal{W} , k(P_\mathcal{W} )})/2 \simeq \CH_0(\mathcal{W} _{k(\mathcal{W} )})/2$. This isomorphism maps the diagonal class to the diagonal class. From this we conclude that \eqref{eq:diagonalclass} is contained in the image of the map \eqref{eq:PWBaseChangeMap} if and only if $\delta_\mathcal{W}  - z_{k(\mathcal{W} )}$ is contained in the image of 
    \begin{equation}
        \label{eq:WBaseChangeMap}
        \CH_0(\mathcal{W} )/2 \longrightarrow \CH_0(\mathcal{W} _{k(\mathcal{W} )})/2.
    \end{equation}
    By construction $\mathcal{W} $ is the same variety as the one considered in \cite[Lemma 3.4]{Ska23}, which in turn is birational to the quadric bundle constructed in \cite{HPT18} and therefore has a nonzero unramified cohomology class $\alpha \in H_{nr}^i(k(\mathcal{W})/k, \Z/2)$, see \cite[Corollary 3.8]{Ska23}.
    
    The main result in \cite{Ska23} is proven by showing that the cycle $\delta_\mathcal{W}  - z_{k(\mathcal{W} )}$ is not in the image of $\CH_0(\mathcal{W} ) \to \CH_0(\mathcal{W} _{k(\mathcal{W})})$. Our goal here is to prove that the cycle is still not in the image after the reduction mod 2. Luckily, this introduces very little additional complications. In fact, it is straightforward to use linearity of the Merkurjev paring to reduce the argument to the same as in \cite{Ska23}. We sketch the entire argument here for ease of reference.

    The central idea is that if $\delta_\mathcal{W}  - z_{k(\mathcal{W} )}$ is in the image of  \eqref{eq:WBaseChangeMap}, then we have an equality 
    \begin{equation}
        \label{eq:DiagonalEquals1}
        \delta_\mathcal{W}  = z_{k(\mathcal{W} )} + 2z' \in \CH_0(\mathcal{W} _{k(\mathcal{W} )}),
    \end{equation}
    where $z$ is a zero-cycle on $\mathcal{W} $ and $z'$ a zero-cycle on $\mathcal{W} _{k(\mathcal{W} )}$.
    
    Intuitively, we would like to say that the Merkurjev pairing of one side with $\alpha$ vanishes, whereas pairing $\alpha$ with the other side does not. Thus we obtain a contradiction.

    However $\mathcal{W} $ is singular, and the Merkurjev pairing is only defined on smooth varieties. To address this issue we will first have to blow up a subvariety of $\mathcal{W} $ as a first step towards resolving the singularities. Having a very explicit description of this first step turns out to be crucial for the argument. After this blow-up, we can use an alteration and a general result by Schreieder (\Cref{lem:GenericfibreSmooth}) to deal with the remaining singularities. Throughout we must keep track on how each of these operations change the equality \eqref{eq:DiagonalEquals1}.

    Following an idea from \cite[Theorem 7.1]{NO22}, we observe that $\mathcal{W} $ is the intersection of a cubic hypersurface containing the plane $D$ defined by $x_3=x_4=x_5=x_6=0$ and a cone $Q$ over a quadric surface with vertex plane $D$. By blowing up this plane we get a variety $\mathcal{W} ' = \blowup_D \mathcal{W} $, with a map to $\projspace^1 \times \projspace^1$ induced by the projection of the $\projspace^3$-bundle $\blowup_D Q$ to the quadric surface. As is remarked in \cite{Ska23}, $\mathcal{W} '$ is not smooth, but the generic fibre of the quadric bundle is smooth. We can now choose an alteration $\tau \colon \widetilde{\mathcal{W} } \to \mathcal{W} '$ of odd degree.

    Before we use the Merkurjev pairing on $\widetilde{\mathcal{W} }$, we must understand what equality in $\CH_0(\widetilde{\mathcal{W} } \times k(\mathcal{W}))$ follows from \eqref{eq:DiagonalEquals1}. Since the blowup $\mathcal{W} ' \to \mathcal{W} $ is an isomorphism away from the singular locus $E$, we have an equality
    \begin{equation}
        \label{eq:DiagonalEquals2}
        \delta_\mathcal{W}  = z_{k(\mathcal{W} )} + 2z' + z'' \in \CH_0(\mathcal{W} '_{k(\mathcal{W} )}),    
    \end{equation}
    where $z''$ is supported on $E$. Next, let $O'$ be the complement of the singular locus of $\mathcal{W} '$, and $\widetilde{O} \coloneqq \tau^{-1}(O') \subset \widetilde{\mathcal{W} }$. Then $\tau$ induces a well-defined pullback map from $\CH_0(O') \to \CH_0 (\widetilde{O})$. 
    
    So the pullbacks of the two sides of \eqref{eq:DiagonalEquals2} by $\tau$ are equal on $\widetilde{O}$, hence we obtain an equality 
    \begin{equation}
        \label{eq:DiagonalEquals3}
        \tau^*(\delta_{\mathcal{W} '})= \widetilde{\delta}_{\tau} = \tau^* z_{k(\mathcal{W} )} + 2 \tau^*z' + \tau^*z'' + z''' \in \CH_0(\mathcal{W} '_{k(\mathcal{W} )}),
    \end{equation}
    where $z'''$ is supported on $\widetilde{\mathcal{W} } \setminus \widetilde{O}$.

    We will now compute that the Merkurjev pairing of the nonzero class $\alpha$ with both sides of \eqref{eq:DiagonalEquals3} is different. Hence no equality of the form \eqref{eq:DiagonalEquals3} is possible, proving the proposition.

    Since $\tau$ is \'etale in a neighborhood of the diagonal point, we have $\tau^*(\delta_{\mathcal{W} '}) = \widetilde{\delta}_{\tau}$, where $\widetilde{\delta}_\tau$ is the $0$-cycle corresponding to the graph of the map $\tau$ in $\widetilde{\mathcal{W} } \times \mathcal{W} '$. This graph is isomorphic to $\widetilde{\mathcal{W} }$, hence $\tau$ induces a map from $\Spec k(\widetilde{\mathcal{W} })$ to $\Spec k(\mathcal{W} ) $. Furthermore, to compute the pairing, we can compute the pushforward of $\tau^* \alpha$ by this map. We get
    \[\innerpro{\widetilde{\delta}_{\tau}}{\tau^*\alpha} = \tau_* \tau^* \alpha = (\deg \tau) \alpha \neq 0.\]
    Since $\alpha$ is nonzero of even order, this class is also nonzero.

    We now compute the pairing of $\tau^* \alpha$ with the summands on the right-hand side. First we note that since $\alpha$ has order 2, so does $\tau^* \alpha$. By linearity, the pairing with $2 \tau^*z'$ must be zero.

    Next we look at the pairing
    \[\innerpro{\tau^*z_{k(\mathcal{W} )}}{\tau^*\alpha}= 0.\]
    By definition of the Merkurjev pairing it factors through the restriction of $\tau^*\alpha$ to a closed point on $\widetilde{\mathcal{W} }$, a smooth variety over an algebraically closed field, and the restriction of unramified cohomology classes of positive degree to such classes vanishes. Hence the pairing is zero.

    We next consider
    \[ \innerpro{\tau^*z''}{\tau^* \alpha} = 0. \]
    By functoriality of pullback of unramified cohomology, we can compute the pairing after restricting the unramified cohomology class to the smooth locus $E \cap O' \subset \mathcal{W} '$ and then pulling back to $\widetilde{\mathcal{W} }$. One can check that $E$ is the conic bundle corresponding to the class $\alpha$, hence the restriction of $\alpha$ to $E$ vanishes. (See  \cite[Lemma 3.9]{Ska23}.) We conclude that $\innerpro{\tau^*z''}{\tau^* \alpha} = 0$.

    It remains to prove that 
    \[\innerpro{z'''}{\tau^*\alpha} = 0.\]
    But $z'''$ is supported on $\widetilde{\mathcal{W} } \setminus \widetilde{O}$, which is the inverse image of the singular locus of $\mathcal{W} '$. Since the singular locus of $\mathcal{W} '$ does not dominate $\projspace^1 \times \projspace^1$, we can conclude by \Cref{lem:GenericfibreSmooth} below.

From these computations, we see that the pairing of $\tau^*\alpha$ with the left hand side of \eqref{eq:DiagonalEquals3} is nonzero, but the pairing of $\tau^*\alpha$ with the right-hand side of \eqref{eq:DiagonalEquals3} is zero. Since \eqref{eq:DiagonalEquals3} does not hold, neither does \eqref{eq:DiagonalEquals1}, hence $\Phi$ is not surjective.
\end{proof}

\begin{theorem}{{\cite[Theorem 9.2]{Sch19b}}}	
  \label{lem:GenericfibreSmooth}
  Let $f \colon Y \to S$ be a surjective morphism of proper varieties over an algebraically closed field $k$ with $char(k) \neq 2$ whose generic fibre is birational to a smooth quadric over $k(S)$. Let $n = \dim(S)$ and assume that there is a class $\alpha \in H^n(k(S),\Z/2)$ with $f^* \alpha \in H_{nr}^n(k(Y)/k,\Z/2)$.
  Then for any dominant generically finite morphism $\tau \colon Y' \to Y$ of varieties, and for any subvariety $E \subset Y'$ that meets the smooth locus of $Y'$ and which does not dominate $S$ via $f \circ \tau$, we have $\restr{(\tau^* f^* \alpha)}_{E} = 0 \in H^n(k(E),\Z/2)$.
\end{theorem}

    \subsection{Proof of the main result}
    \begin{theorem}\label{thm:GeomGenFibreNoDecompOfDiag}
        The geometric generic fibre of the family $\tilde{\mathcal{X}} \to \Spec k[[t]]$ from \Cref{lem:TildeCurlyX} does not admit a decomposition of the diagonal, and is therefore not retract rational.
    \end{theorem}
    \begin{proof}
        Assume for contradiction that the geometric generic fibre of $\tilde{\mathcal{X}}$ admits a decomposition of the diagonal. Then the map $\Phi_{\Tilde{\mathcal{X}}}$ is universally surjective modulo 2 by \Cref{rem:SlightSimplificationOfPavicSchreieder}. In particular, the modulo $2$ reduction of the class $\delta_{P_W} - z_{k(P_W)}$ is contained in the image of the base change of $\Phi_{\Tilde{\mathcal{X}}}$ to $k(P_W)$, the function field of $P_W$. Since the map $\Phi_{\mathcal{X}}$ commutes with specialization of the involved varieties, and this specialization preserves the diagonal class, the class $\delta_{P_\mathcal{W}} - z_{k(P_\mathcal{W})}$ is preserved by the specialization (see also \Cref{prop:KeyReduction}) and must be contained in the image of the specialization of $\Phi_{\Tilde{\mathcal{X}}}$. But this contradicts \Cref{lem:MerkurjevObstruction}, so we conclude that the geometric generic fibre does not admit a decomposition of the diagonal, which in turn implies that it is not retract rational.
    \end{proof}

    \begin{cor}\label{cor:MainResult}
        Let $k$ be an algebraically closed field of characteristic different from $2$ with prime field $F$. Assume that the transcendence degree of $k$ over $F$ $\trdeg_F k \geq 3$. Then the very general (3,3) fivefold over $k$ does not admit a decomposition of the diagonal, and is therefore not retract rational.
    \end{cor}
    \begin{proof}
        Fix an inclusion $\overline{F}(\alpha,\beta,\gamma) \subset k$ and consider the parameter space $B$ of smooth (3,3) fivefolds over $k$. By \Cref{thm:GeomGenFibreNoDecompOfDiag} there exists a closed point $0 \in B$ such that the (3,3) fivefold $X_0$ admits no decomposition of the diagonal. Since the very general (3,3) fivefold specializes to $X_0$ (see \Cref{lem:VeryGeneralSpecializesToOne}), we find that the very general (3,3) fivefold over $k$ admits no decomposition of the diagonal by \cite[Theorem 1.14]{CTP16}.
    \end{proof}

    \begin{remark}
        The bound on the transcendence degree corresponds to the number of specializations used to control the Chow group of the special fiber of the semi-stable specialization in \Cref{lem:TildeCurlyX}. The construction presented here uses three steps to achieve this goal hence the bound in \Cref{cor:MainResult}. Roughly speaking, we use the first two steps to simplify the Chow groups of the components $Y$ and $Z$, cf. \eqref{eq:strategydiagram}. The final step is then to specialize the final component of the fiber to a variety with a nontrivial unramified cohomology class, while maintaining control of the resulting Chow groups.
        There might be other constructions achieving this goal in fewer steps, hence lowering the bound on the transcendence degree, but it seems likely that such a construction would require a different and more careful argument to control the Chow groups. 
    \end{remark}
    \printbibliography
\end{document}